\documentclass[12pt]{article}

\usepackage{amsmath,amsthm}
\usepackage{amssymb}

\usepackage{enumitem}

\usepackage[T1]{fontenc}
\usepackage[all]{xy}

\newtheorem{theorem}{Theorem}[section]
\newtheorem{corollary}[theorem]{Corollary}
\newtheorem{lemma}[theorem]{Lemma}
\newtheorem{proposition}[theorem]{Proposition}

\theoremstyle{definition}
\newtheorem{definition}[theorem]{Definition}
\newtheorem{remark}[theorem]{Remark}
\newtheorem{example}[theorem]{Example}

\numberwithin{equation}{section}

\def\Z{\mathbb{Z}}
\def\Q{\mathbb{Q}}
\def\F{\mathbb{F}}
\newcommand{\vphi}{{\varphi}}
\newcommand{\p}{{\mathfrak p}}
\newcommand{\q}{{\mathfrak q}}
\renewcommand{\frm}{{\mathfrak m}}
\newcommand{\frn}{{\mathfrak n}}

\newcommand{\frM}{{\mathfrak M}}

\newcommand{\cO}{{\mathcal O}}
\newcommand{\set}[1]{\left\lbrace#1\right\rbrace }
\newcommand{\Span}[1]{\left<#1\right>}
\DeclareMathOperator{\Frac}{Frac}
\DeclareMathOperator{\ord}{ord}

\begin{document}

%%%%% To ease editing, for IMPAN journals add:

%\baselineskip=17pt

%%%%%%%%%%%%%%%%

\title{Super-multiplicativity of ideal norms in number fields}

\author{Stefano Marseglia}
\AtEndDocument{\vspace{2cm}\noindent
Department of Mathematics\\
Stockholm University\\
SE - 106 91 Stockholm, Sweden.\\
e-mail: stefanom@math.su.se
}

\date{}

\maketitle

%% Classification and key words; note that the 2010 classification is used:

\renewcommand{\thefootnote}{}



\renewcommand{\thefootnote}{\arabic{footnote}}
\setcounter{footnote}{0}

%%%%%%%%

\begin{abstract}
In this article we study inequalities of ideal norms. We prove that in a subring $R$ of a number field every ideal can be generated by at most $3$ elements if and only if the ideal norm satisfies $N(IJ)\geq N(I)N(J)$ for every pair of non-zero ideals $I$ and $J$ of every ring extension of $R$ contained in the normalization $\tilde R$.
\end{abstract}

\section{Introduction}
When we are studying a \emph{number ring} $R$, that is a subring of a number field $K$, it can be useful to understand the size of its ideals compared to the whole ring. The main tool for this purpose is the norm map which associates to every non-zero ideal $I$ of $R$ its index as an abelian subgroup $N(I)=[R:I]$.
If $R$ is the \emph{maximal order}, or \emph{ring of integers}, of $K$ then this map is multiplicative, that is, for every pair of non-zero ideals $I,J\subseteq R$ we have $N(I)N(J)=N(IJ)$. If the number ring is not the maximal order this equality may not hold for some pair of non-zero ideals. For example, if we consider the quadratic order $\Z[2i]$ and the ideal $I=(2,2i)$, then we have that $N(I)=2$ and $N(I^2)=8$, so we have the inequality $N(I^2)> N(I)^2$.
Observe that if every maximal ideal $\p$ of a number ring $R$ satisfies $N(\p^2)\leq N(\p)^2$, then we can conclude that $R$ is the maximal order of $K$ (see Corollary \ref{cor:submdedekind}).

In Section \ref{sec:prelim} we recall some basic commutative algebra and algebraic number theory and we apply them to see how the ideal norm behaves in relation to localizations and ring extensions.

In Section \ref{sec:quadquartcases} we will see that the inequality in the previous example is not a coincidence. More precisely, we will prove that in any quadratic order we have $N(IJ)\geq N(I)N(J)$ for every pair of non-zero ideals $I $ and $J$.
We will say that the norm is \emph{super-multiplicative} if this inequality holds for every pair of non-zero ideals (see Definition \ref{def:sm}).
We will show that this is not always the case by exhibiting an a order of degree $4$ where we have both (strict) inequalities, see Example \ref{ex:degree4}.

In a quadratic order every ideal can be generated by $2$ elements and in a order of degree $4$ by $4$ elements, so we are led to wonder if the behavior of the norm is related to the number of generators and what happens in a cubic order, or more generally in a number ring in which every ideal can be generated by 3 elements.

The main result of this work is the following:
\begin{theorem}
\label{mainthm}
Let $R$ be a number ring. The following statements are equivalent:
\begin{enumerate}[label=\upshape(\roman*), leftmargin=*, widest=iii]
\item \label{impl:1} every ideal of $R$ can be generated by $3$ elements;
\item \label{impl:2} every ring extension $R'$ of $R$ contained in the normalization of $R$ is super-multiplicative.
\end{enumerate}
\end{theorem}

Theorem \ref{mainthm} is an immediate consequence of the following two stronger results, which are proved respectively in Sections \ref{sec:firstimpl} and \ref{sec:secondimpl}.
\begin{theorem}
 \label{thm:firstimpl}
 Let $R$ be a commutative Noetherian domain of dimension $1$ where every ideal can be generated by $3$ elements.
 Then $R$ is super-multiplicative.
 Moreover, every ring extension $R'$ of $R$  such that the additive group of $R'/R$ has finite exponent is also super-multiplicative.
\end{theorem}
\begin{theorem}
 \label{thm:secondimpl}
 Let $R$ be a number ring with normalization $\tilde{R}$ such that for every maximal $R$-ideal $\frm$ the ideal norm of the number ring $R+\frm\tilde R$ is super-multiplicative. Then every $R$-ideal can be generated by $3$ elements.
\end{theorem}

\section{Preliminaries}
\label{sec:prelim}
A field $K$ is called \emph{number field} if it is a finite extension of $\Q$. In this article all rings are unitary and commutative. We will say that $R$ is a \emph{number ring} if it is a subring of a number field. A number ring for which the additive group is finitely generated is called an \emph{order} in its field of fractions. In every number ring there are no non-zero additive torsion element. Every order is a free abelian group of rank $[\Frac(R):\Q]$, where $\Frac(R)$ is the fraction field of $R$.

\begin{proposition}
Let $R$ be a number ring. Then
\begin{enumerate}
\item every non-zero $R$-ideal has finite index;
\item $R$ is Noetherian;
\item if $R$ is not a field then it has Krull dimension $1$, that is every non-zero prime ideal is maximal;
\item if $S$ is a number ring containing $R$ and $\p$ a maximal ideal of $R$, then there are only finitely many prime $S$-ideal $\q$ above $\p$, that is $\q\supseteq \p S$;
\item $R$ has finite index in its normalization $\tilde R$.
\end{enumerate}
\end{proposition}

For a proof and more about number rings see \cite{psh08}.

Recall that for a commutative domain $R$ with field of fractions $K$, a \emph{fractional $R$-ideal} $I$ is a  non-zero $R$-submodule of $K$ such that $xI\subseteq R$ for some non-zero $x\in K$. Multiplying by a suitable element of $R$, we can assume that the element $x$ in the definition is in $R$. It is useful to extend the definition of the index to arbitrary fractional ideals $I$ and $J$ taking:
\[[I:J]=\dfrac{[I:I\cap J]}{[J:I\cap J]}.\]
It is an easy consequence that we have $[I:J]=[I:H]/[J:H]$ for every fractional ideal $H$.
In particular, if $[R:I]$ is finite we call it the \emph{norm of the ideal $I$}, and we denote it $N(I)$. In general the ideal norm is not multiplicative.
\begin{lemma} 
\label{lemma:principalideal}
Let $R$ be a number ring and let $I$ be a non-zero $R$-ideal. For every non-zero $x\in K$ we have 
\[ N(xR)N(I)=N(xI).\]
\end{lemma}
\begin{proof}
As $R$ is a domain, the multiplication by $x$ induces an isomorphism $R/I\simeq xR/xI$ of (additive) groups. Hence we have $[R:xR]=[I:xI]$ and therefore $[R:xR][R:I]=[R:xI]$.
\end{proof}

\begin{proposition}
\label{prop:lengthproduct}
 Let $S\subseteq R$ be an extension of commutative rings. Let $I$ be an $R$-ideal such that $[R:I]$ is finite. Then
 \[[R:I] = \prod_{\frm} [R_\frm: I_\frm] = \prod_\frm \# ( S / \frm )^{l_{S_\frm}( R_\frm / I_\frm )},\]
 where the products are taken over the maximal ideals of $S$ and $l_{S_\frm}$ denotes the length as an $S_\frm$-module.
 Moreover, we have that 
 \[l_S(R/I)=\sum_\frm l_{S_\frm}(R_\frm/I_\frm).\]
\end{proposition}
\begin{proof}
 As $R/I$ has finite length as an $S$-module, there exists a composition series
 \[R/I=M_0\supset M_1\supset \cdots \supset M_l = 0,\]
 where the $M_i$ are $S$-modules such that $M_i/M_{i+1}\simeq S/\frm_i$ for some maximal $S$-ideal $\frm_i$. Now fix a maximal $S$-ideal $\frm$. Observe that $\#\set{i\ :\ \frm_i=\frm} = l_{S_\frm}(R_\frm/I_\frm)$ because all the factors isomorphic to $S/\frm_i$ disappear if we localize at $\frm\neq \frm_i$. This implies that $l_S(R/I)=\sum_\frm l_{S_\frm}(R_\frm/I_\frm)$ and that $[ R_\frm : I_\frm ] =  \# ( S / \frm  )^{l_{S_\frm}( R_\frm/I_\frm)}$. By \cite[2.13, p.72]{eis95} we have $[R:I] = \prod_{\frm} [R_\frm:I_\frm]$.
 Observe that there is no harm in taking the product over all the maximal ideal of $S$ because the module $R/I$ vanishes if we localize at a maximal ideal that does not appear in its composition series.
\end{proof}

\begin{proposition}
\label{prop:normmult}
Let $R$ be a number ring, $I$ an invertible $R$-ideal. Then for every $R$-ideal $J$ we have
\[N(I)N(J)=N(IJ).\]
\end{proposition}
\begin{proof}
Recall that if an ideal $I$ is invertible then the localization $I_\frm$ at every maximal  $R$-ideal $\frm$ is a principal $R_\frm$-ideal (see \cite[11.3, p.80]{mats89}). So by Lemma \ref{lemma:principalideal} we have that $[R_\frm:J_\frm][R_\frm:I_\frm]=[R_\frm:(IJ)_\frm]$ for every $\frm$. Hence by Proposition \ref{prop:lengthproduct}
\[N(IJ)=\prod_\frm [R_\frm:(IJ)_\frm] = \prod_\frm [R_\frm:I_\frm]\prod_\frm [R_\frm:J_\frm]=N(I)N(J).\]
\end{proof}

\begin{proposition}
\label{prop:lengthlocal}
 Let $S\subseteq R$ be an extension of commutative domains, $\frm$ a maximal $S$-ideal and $J$ a proper ideal of the localization $R_\frm$ such that $R_\frm/J$ has finite length as an $S_\frm$-module. Then
 \[ \frac{R}{J\cap R} \simeq \frac{R_\frm}{J} \]
 as $S$-modules. Moreover,
 \[l_S\left( \frac{R}{J\cap R} \right) = l_{S_\frm}\left( \frac{R_\frm}{J} \right). \]
\end{proposition}
\begin{proof}
 As $R$ is a domain the localization morphism $R \to R_\frm$ composed with the projection $R_\frm \to R_\frm/J$ induces an injective morphism $R/(J\cap R)\to R_\frm/J$. As $l_{S_\frm}( R_\frm / J)$ is finite, $R/(J\cap R)$ is annihilated by some power of $\frm$ and by \cite[2.13, p.72]{eis95} we have that it is isomorphic to its localization at $\frm$. As $(J\cap R)_\frm = J$ we have that $R/(J\cap R)\simeq R_\frm/J$ as $S$-modules. In particular they have the same length as $S$-modules. By Proposition \ref{prop:lengthproduct} we have that $l_S(R_\frm/J) = \sum_\frn l_{S_\frn}((R_\frm/J)_\frn)$, where the sum is taken over the maximal $S$-ideals.
 So to conclude, we need to prove that if $\frn\neq \frm $, then $l_{S_\frn}((R_\frm/J)_\frn)=0$, which is a direct consequence of the fact that $(R_\frm/J)_\frn=0$ when $\frn \neq \frm$.
\end{proof}

\begin{definition}
\label{def:sm}
Let $R$ be a commutative ring. We will say that the ideal norm of $R$ is \emph{super-multiplicative}  if for every pair of $R$-ideals $I$ and $J$ such that $[R:IJ]$ is finite we have
\[N(IJ)\geq N(I)N(J).\]
For brevity we will say that $R$ is \emph{super-multiplicative}.
\end{definition}
\begin{proposition}
 Let $R$ be a number ring. Let $I$ be any non-zero $R$-ideal and $\p$ a maximal $R$-ideal. Then $N(\p I)\geq N(I)N(\p)$.
\end{proposition}
\begin{proof}
By the isomorphism of abelian groups
\[\dfrac{R/\p I}{I/\p I}\simeq R/I\]
we get that
\[\#(I/\p I)=N(\p I)/N(I).\]
Since $I/\p I$ is a $(R/\p)$-vector space of finite dimension, say $d$, we have $\#(I/\p I)=N(\p)^d$.
Therefore $N(\p I)=N(I)N(\p)^d\geq N(I)N(\p)$.
\end{proof}
\begin{corollary}
\label{cor:submdedekind}
 Let $R$ be a number ring. Assume that for every maximal $R$-ideal $\p$ we have the inequality $N(\p^2)\leq N(\p)^2$. Then $R$ is a Dedekind domain.
\end{corollary}
\begin{proof}
 As in the proof of the previous Proposition we obtain $N(\p^2)=N(\p)N(\p)^d,$ where $d=\dim_{R/\p}(\p/\p^2)$. Using the hypothesis we get $N(\p)^d\leq N(\p)$ which implies that $d\leq 1$. Observe that it cannot be zero, as $\p^2\subsetneq \p$. Hence we have that $\dim_{R/\p}(\p/\p^2) = 1$
 for every maximal ideal $\p$, which is equivalent to say that $R$ is a Dedekind domain.
\end{proof}

Being super-multiplicative is a local property for commutative domains. More precisely:
\begin{lemma}
\label{lemma:smlocal}
Let $S\subseteq R$ be an extension of commutative domains. Then $R$ is super-multiplicative if and only if for every maximal $S$-ideal $\frm$ we have that $R_\frm$ is super-multiplicative.
\end{lemma}
\begin{proof}
Assume that $R$ is super-multiplicative and let $I$ and $J$ be $R_\frm$-ideals with $IJ$ of finite index in $R_\frm$. Then by Proposition \ref{prop:lengthlocal} we have $[R_\frm:IJ]=[R:IJ\cap R]$, $[R_\frm:I]=[R:I\cap R]$ and $[R_\frm:J]=[R:J\cap R]$.
By Proposition \ref{prop:lengthproduct}, we obtain $l_{S_\frn}((R/IJ\cap R)_\frn)=0$, $l_{S_\frn}((R/I\cap R)_\frn)=0$ and $l_{S_\frn}((R/J\cap R)_\frn)=0$, for every maximal $S$-ideal $\frn$ distinct from $\frm$.
We have that $(I\cap R)_\frm(J\cap R)_\frm = (IJ\cap R)_\frm = IJ$ and $(I\cap R)_\frn(J\cap R)_\frn = (IJ\cap R)_\frn = R_\frn$ for $\frn $ a maximal $S$-ideal of $S$ distinct from $ \frm$, so $(I\cap R)(J\cap R) = (IJ\cap R)$. Hence we get that $[R_\frm:IJ]\geq [R_\frm:I][R_\frm:J]$, that is $R_\frm$ is super-multiplicative.\\
In the other direction, if we have that $R_\frm$ is super-multiplicative for every $\frm$, taking the product of the norms of the localizations leads to the required global inequality by Proposition \ref{prop:lengthproduct}.
\end{proof}
The next result is well known. We include a proof for sake of completeness.
\begin{proposition}
\label{prop:semilocal}
Let $R$ be a \emph{semilocal} commutative domain, i.e.~a domain with a finite number of maximal ideals. Then, a fractional ideal of $R$ is invertible if and only if it is principal and non-zero. In particular, a semilocal Dedekind domain, like the normalization of any local number ring, is a principal ideal domain.
\end{proposition} 
\begin{proof}
One direction of the proof is trivial, because if $x\in R$ is non-zero, then the ideal $(x)$ has inverse $(x^{-1})$. Let's prove the other implication. Let $I$ be a fractional $R$-ideal. Multiplying by an appropriate element of the fraction field of $R$, we can assume that $I\subseteq R$. Observe that this doesn't affect the number of generators. Suppose that $I$ is an invertible $R$-ideal, with inverse $J$, i.e.~$IJ=R$. Let $\frm_1,\cdots,\frm_l$ be the maximal ideals of $R$. As $IJ\nsubseteq \frm_k$ for every $k$, there exist $a_k\in I,b_k\in J$ such that $a_kb_k\in R\setminus \frm_k$.
By the Chinese Remainder Theorem, for every $k$ there exists an element $\lambda_k\not\in \frm_k$ and $\lambda_k\in \frm_j$ for every $j\neq k$.
Then define
\[
a = \lambda_1a_1+\cdots+\lambda_la_l\in I,\quad b=\lambda_1b_1+\cdots+\lambda_l b_l\in J
\]
and consider the product:
\[ab = \sum_{1\leq i,j\leq l}\lambda_i\lambda_ja_ib_j.\]
Observe that $\lambda_i\lambda_ja_ib_j\not\in \frm_k$ if and only if $i=j=k$. Hence $ab\not\in\frm_k$ for every $k$ and it must therefore be a unit. Then
\[(a)\subseteq I=abI\subseteq aJI=aR=(a)\]
as required.
\end{proof}

The following version of the Nakayama Lemma is classical and will be used several times in the rest of the paper without mentioning.
\begin{lemma}[{\cite[Proposition 2.8]{AtiyahMacdonald69}}]
\label{lemma:Nakayama}
Let $R$ be a local ring with maximal ideal $\frm$ and let $M$ be a finitely generated $R$ module. Let $x_1,\ldots,x_n$ be elements of $M$ whose images in $M/\frm M$ form a basis over $R/\frm$. Then the $x_i$ generate $M$.
\end{lemma}

%%%%%%%%%%%%%%%%%%%%%%%%%%%%%%%%%%%%%%%%%%%%%%%%%%%%%%%%%%%%%%%%%%%%%%%%%%%%%%%%%%%%%%%%%%%%%%%%%%%%%%%%%%%%%%%%%%%%%%%%%%%%%
%%%%%%%%%%%%%%%%%%%%%%%%%%%%%%%%%%%%%%%%%%%%%%%%%%%%%%%%%%%%%%%%%%%%%%%%%%%%%%%%%%%%%%%%%%%%%%%%%%%%%%%%%%%%%%%%%%%%%%%%%%%%%
%%%%%%%%%%%%%%%%%%%%%%%%%%%%%%%%%%%%%%%%%%%%%%%%%%%%%%%%%%%%%%%%%%%%%%%%%%%%%%%%%%%%%%%%%%%%%%%%%%%%%%%%%%%%%%%%%%%%%%%%%%%%%
%%%%%%%%%%%%%%%%%%%%%%%%%%%%%%%%%%%%%%%%%%%%%%%%%%%%%%%%%%%%%%%%%%%%%%%%%%%%%%%%%%%%%%%%%%%%%%%%%%%%%%%%%%%%%%%%%%%%%%%%%%%%%
%%%%%%%%%%%%%%%%%%%%%%%%%%%%%%%%%%%%%%%%%%%%%%%%%%%%%%%%%%%%%%%%%%%%%%%%%%%%%%%%%%%%%%%%%%%%%%%%%%%%%%%%%%%%%%%%%%%%%%%%%%%%%

\section{Quadratic and degree 4 case}
\label{sec:quadquartcases}
In this section we prove that every quadratic order is super-multiplicative. This result is a consequence of Theorem \ref{mainthm} stated in the introduction. We report this particular case separately because the argument of the proof is different and of its own interest. We will also exhibit in the end of this section an example that shows that an analogous theorem is not true for orders in a number field of degree $4$.

\begin{lemma}
\label{lemma:quadmult}
	Let $R$ be an order in a quadratic field $K$. Let $I$ be a non-zero $R$-ideal and $R_I$ its multiplier ring, i.e.~$R_I=\set{x\in K:\ xI\subseteq I}$. Then $I$ is an invertible ideal of $R_I$.
\end{lemma}
\begin{proof}
	By \cite[11.3, p.80]{mats89} it suffices to show that the localization of $I$ at every maximal ideal $\p$ of $R_I$ is principal.
	Assume that this is not the case, say that $I_\p$ is not principal.
	Observe that if $\p$ is above the rational prime $p$ we cannot have $pR_I=\p$, because $R_{I\p}$ would be a DVR and $I_\p$ would be invertible. 
	As $ [R_I:pR_I]=p^2 $ and $ pR_I \subsetneq \p $, we have $[\p : pR_I ]= [R_I:\p] =p$ and $R_I/\p\simeq \F_p$.
	By Lemma~\ref{lemma:principalideal} we have $[I:pI]=[R_I:pI]/[R_I:I]=[R_I:pR_I]=p^2$. As $I_\p$ is not principal, by Nakayama's Lemma we have that $I_\p/\p I_\p\simeq I/\p I$ is a $R_{I\p}/\p$-vector space of dimension $2$. Hence also $[I:\p I]=p^2$
	which implies $pI=\p I$ because $pI\subseteq \p I$ and they have the same index in $I$. So by the definition of multiplier ring  $p^{-1}\p\subseteq R_I$, hence $pR_I=\p$ by the maximality of $\p$. Contradiction.
	So $I$ is an invertible $R_I$-ideal.
\end{proof}

\begin{lemma}
\label{lemma:cycicpgroup}
Let $R$ be a quadratic order with integral closure $\tilde R$ and consider the localizations at a prime number $p\in \Z$, namely $\tilde{R}_{(p)}=\tilde{R}\otimes \Z_{(p)}$ and $R_{(p)}=R\otimes \Z_{(p)}$.  Then we have that $\tilde{R}_{(p)}/R_{(p)}\simeq \Z/p^n\Z$ for some $n\in \Z_{\geq 0}$. 
\end{lemma}
\begin{proof}
Note that $\tilde{R}/R$ is a finite abelian group which can be decomposed in the product of finitely many cyclic groups with order a prime power. When we localize at $p$ we consider only the $p$-part of this decomposition. As $R$ is quadratic what is left is a cyclic group.
\end{proof}

\begin{theorem}
	The ideal norm in a quadratic order is super-multiplicative.
\end{theorem}
\begin{proof}
	Let $R$ be a quadratic order and $I,J$ two non-zero ideals of $R$. We want to show that
	\[\dfrac{[R:IJ]}{[R:I][R:J]}\geq 1.\]
	Let $p$ be an arbitrary rational prime, we want to prove that
	\[\dfrac{[R_{(p)}:I_{(p)}J_{(p)}]}{[R_{(p)}:I_{(p)}][R_{(p)}:J_{(p)}]}\geq 1. \label{local rel}\tag{*}\]
	By Lemma \ref{lemma:quadmult} we have that $I$ (resp.~$J$) is invertible in its multiplier ring $R_I$ (resp.~$R_J$). Note that if $\q$ is a maximal $R_I$-ideal above the rational prime $q$, then $\q\cap (\Z\setminus (p))=(q)\setminus (p)$ which is empty if and only if $p=q$. This means that the maximal ideals of $R_{I(p)}$ are exactly the ones above $p$ and similarly for $R_{J(p)}$. So the localization of $I$ (resp.~$J$) at $(p)$ is a principal ideal of $R_{I(p)}$ (resp.~$R_{J(p)}$) by Proposition \ref{prop:semilocal} . Say that we have $I_{(p)}=xR_{I(p)}$ and $J_{(p)}=yR_{J(p)}$ and observe that $R_{I(p)}$ and $R_{J(p)}$ are both $R_{(p)}$-fractional ideals.  So by Lemma \ref{lemma:principalideal} 
	\begin{align*}
	[R_{(p)}:I_{(p)}] &=[R_{(p)}:R_{I(p)}][R_{(p)}:xR_{(p)}],\\
	[R_{(p)}:J_{(p)}] &=[R_{(p)}:R_{J(p)}][R_{(p)}:yR_{(p)}],\\
	[R_{(p)}:I_{(p)}J_{(p)}]&=[R_{(p)}:xR_{I(p)}yR_{J(p)}]=\\
	&=[R_{(p)}:R_{I(p)}R_{J(p)}][R_{(p)}:xR_{(p)}][R_{(p)}:yR_{(p)}].
	\end{align*}
	If we substitute these equalities in (\ref{local rel}) we get:	
	\[\dfrac{[R_{(p)}:R_{I(p)}R_{J(p)}]}{[R_{(p)}:R_{I(p)}][R_{(p)}:R_{J(p)}]} = \dfrac{[R_{I(p)}:R_{(p)}][R_{J(p)}:R_{(p)}]}{[R_{I(p)}R_{J(p)}:R_{(p)}]}.\]
	As $\tilde{R}_{(p)}/R_{(p)}$ is a cyclic $p$-group by Lemma \ref{lemma:cycicpgroup},  the lattice of its subgroups is totally ordered w.r.t.~the inclusion relation. Then as $R \subseteq R_I,R_J\subseteq \tilde{R}$, we have that $R_{I(p)}\subseteq R_{J(p)}$ or $R_{J(p)}\subseteq R_{I(p)}$. Assume that the first one holds, then $R_{I(p)}R_{J(p)}=R_{J(p)}$. So we have:
	\[\dfrac{[R_{I(p)}:R_{(p)}][R_{J(p)}:R_{(p)}]}{[R_{I(p)}R_{J(p)}:R_{(p)}]} = \dfrac{[R_{I(p)}:R_{(p)}][R_{J(p)}:R_{(p)}]}{[R_{J(p)}:R_{(p)}]}= [R_{I(p)}:R_{(p)}] \geq 1.\]
	If we have that $R_{J(p)}\subseteq R_{I(p)}$ we proceed in an analogous way.
	As this inequality holds for the localization at every rational prime $p$, by Proposition \ref{prop:lengthproduct} it holds also for the original quotient, hence we get the desired inequality for the global norms.
\end{proof}
	
As we have understood the quadratic case, then we will move to extensions of $\Q$ of higher degree. The next example shows that we cannot prove an analogous theorem for the degree 4 case. 

\begin{example}
\label{ex:degree4}
Consider the field $\Q(\alpha)$, where $\alpha$ is the root of a monic irreducible polynomial of degree $4$ with integer coefficients. Consider the order generated by the ring $\Z$ and the ideal $p\Z[\alpha]$, say $R=\Z + p\Z[\alpha]$, where $p$ is a rational prime number. Take the $R$-ideals $I=pR+ p\alpha R$ and $J=pR + p\alpha ^2R$ and the maximal ideal $M=p\Z[\alpha]$. It's easy to verify that
\begin{align*}
&R = \Z \oplus p\alpha \Z \oplus p\alpha^2 \Z \oplus p\alpha^3 \Z, & & I = p\Z \oplus p\alpha \Z \oplus p^2\alpha^2 \Z \oplus p^2\alpha^3 \Z,\\
&J = p\Z \oplus p^2\alpha \Z \oplus p\alpha^2 \Z \oplus p^2\alpha^3 \Z, & &M = p\Z \oplus p\alpha \Z \oplus p\alpha^2 \Z \oplus p\alpha^3 \Z,\\
&IJ = p^2\Z \oplus p^2\alpha \Z \oplus p^2\alpha^2 \Z \oplus p^2\alpha^3 \Z, & &IM = p^2\Z \oplus p^2\alpha \Z \oplus p^2\alpha^2 \Z \oplus p^2\alpha^3 \Z,
\end{align*}
which gives us $N(I)=N(J)=p^3$, $N(M)=p$, $N(IJ)=N(IM)=p^5$. Hence \[ p^6=N(I)N(J)>N(IJ)=p^5,\qquad p^4=N(I)N(M)<N(IM)=p^5.\]
\end{example}

%%%%%%%%%%%%%%%%%%%%%%%%%%%%%%%%%%%%%%%%%%%%%%%%%%%%%%%%%%%%%%%%%%%%%%%%%%%%%%%%%%%%%%%%%%%%%%%%%%%%%%%%%%%%%%%%%%%%%%%%%%%%%
%%%%%%%%%%%%%%%%%%%%%%%%%%%%%%%%%%%%%%%%%%%%%%%%%%%%%%%%%%%%%%%%%%%%%%%%%%%%%%%%%%%%%%%%%%%%%%%%%%%%%%%%%%%%%%%%%%%%%%%%%%%%%
%%%%%%%%%%%%%%%%%%%%%%%%%%%%%%%%%%%%%%%%%%%%%%%%%%%%%%%%%%%%%%%%%%%%%%%%%%%%%%%%%%%%%%%%%%%%%%%%%%%%%%%%%%%%%%%%%%%%%%%%%%%%%
%%%%%%%%%%%%%%%%%%%%%%%%%%%%%%%%%%%%%%%%%%%%%%%%%%%%%%%%%%%%%%%%%%%%%%%%%%%%%%%%%%%%%%%%%%%%%%%%%%%%%%%%%%%%%%%%%%%%%%%%%%%%%
%%%%%%%%%%%%%%%%%%%%%%%%%%%%%%%%%%%%%%%%%%%%%%%%%%%%%%%%%%%%%%%%%%%%%%%%%%%%%%%%%%%%%%%%%%%%%%%%%%%%%%%%%%%%%%%%%%%%%%%%%%%%%

\section{Proof of Theorem \ref{thm:firstimpl}}
\label{sec:firstimpl}
In this section we introduce a convenient notation for the maximal number of generators for the ideals of a commutative ring and discuss how this quantity behaves when we localize or extend the ring. The rest of the section is devoted to the proof of Theorem \ref{thm:firstimpl}.

\begin{definition}
Let $R$ be a commutative ring. We define
\[ g(R):=\sup_{\substack{I\subset R\\ \text{ideal}}}\Bigl( \inf_{\substack{S\subset I\\ I=\Span{S}}} \#S  \Bigr).  \]
\end{definition}

\begin{remark}
If $R$ is a commutative domain then $g(R)$ is the bound for the cardinality of a minimal set of generators of every fractional ideal $I$. In fact, by the definition of a fractional ideal, there exists a non-zero element $x$ in the fraction field of $R$ such that $xI\subseteq R$. So $xI$ is an $R$-ideal and hence can be generated by $g(R)$ elements, so $I$ can be generated by the same elements multiplied by $x^{-1}$.
\end{remark}

\begin{remark}
\label{rmk:extension}
Let $R\subset R'$ be an extension of commutative domains such that the abelian group $R'/R$ has finite exponent, say $n$. Then we have $g(R')\leq g(R)$. In fact if $J$ is an $R'$-ideal, then $nJ\subseteq R$, and hence $J$ is a fractional $R$-ideal. In particular we are in this situation if $R$ is a number ring and $R'$ is contained in the normalization $\tilde R$ of $R$, because the index $[\tilde R : R]$ is finite.
\end{remark}

\begin{remark}
Let $R$ be a number ring inside a number field $K$. We have $g(R)\leq [K:\Q]$ and this bound is sharp, in the sense that we can find an order $R'$ in $K$ such that $g(R')=[K:\Q]$. Let $\cO_K$ be the maximal order of $K$. Let $I$ be any $R$-ideal. As $R$ is Noetherian, $I$ can be generated by a finite set of elements, say $x_1,\cdots,x_d$. We can find an integer $n\geq 1$ such that $nx_1,\cdots,nx_d\in \cO_K$. Then observe that $I'=nI\cap (\cO_K\cap R)$ is an ideal of $\cO_K \cap R$, so it can be generated over $\Z$ by $[K:\Q]$ elements, say $\alpha_1,\cdots,\alpha_{[K:\Q]}$. As $I'R=nI$, we have that $\alpha_1/n,\cdots,\alpha_{[K:\Q]}/n$ generate $I$ over $R$. Hence $g(R)\leq [K:\Q]$.  To prove the second part, let $\alpha$ be an algebraic integer and put $K=\Q(\alpha)$. Consider $R'=\Z+p\Z[\alpha]$ where $p$ is a rational prime number. Then $\frm=p\Z[\alpha]$ is a maximal ideal of $R'$ and $\dim_{\F_p} \frm/\frm^2=[K:\Q]$, so $g(R')=[K:\Q]$.
\end{remark}

We have a nice description of the behavior of $g(R)$ for a number ring $R$ when we localize at a maximal ideal.
\begin{lemma}
\label{lemma:d-gen}
Let $R$ be a number ring, with normalization $\tilde{R}$. Let $I$ be an $R$-ideal. For every integer $d\geq 2$ the following are equivalent:
\begin{enumerate}
\item \label{d-gen:1} the $R$-ideal $I$ can be generated by $d$ elements;
\item \label{d-gen:2} for every maximal ideal $\p$ of $R$, the $R_\p$-ideal $I_\p$ can be generated by $d$ elements.
\end{enumerate}
\end{lemma}
\begin{proof}
Observe that (\ref{d-gen:1}) implies (\ref{d-gen:2}) is an immediate consequence of the fact that $I_\p = I \otimes_R R_\p$. For the other direction, assume that $I_\p$ is $d$-generated, for every $\p$. We can choose the local generators to be in $I$, just multiplying by the common denominator, which is a unit in $R_\p$.
Now, $\tilde{R}/R$ has finite length as an $R$-module. Consider a composition series
\[ \tilde{R}/R = M_0\supset M_1 \supset \cdots \supset M_l =0. \]
All the factors $M_i/M_{i+1}$ for $i=0,\cdots,l-1$ are simple, hence of the form $R/\p_i$ where $\p_i$ is a maximal $R$-ideal.
If we localize at a maximal ideal $\p\neq \p_i$, for $i=0,\cdots,l-1$, all the factors disappear, and hence we have that $\tilde{R}_\p=R_\p$. Hence $R_\p$ is a local Dedekind domain. Hence $I_\p$ is a principal $R_\p$-ideal.
As the number of factors of the composition series is finite, this situation occurs for almost all the maximal ideals of $R$. In other words we can say that $I/\p I\simeq I_\p/\p I_\p$ is a 1-dimensional $R/\p$-vector space for almost all maximal ideals.
Then consider the finite set $S=\set{\p : \dim_{(R/\p)}I/\p I\neq 1} $.
By the Chinese Remainder Theorem we can pick an element $x_1\in I$ such that $x_1\not\in \p I$ for every $\p\in S$.
Now consider $T=\set{\p : I\supsetneq \p I + (x_1)}$, which is also finite because the ideals $I$ and $(x_1)$ are locally equal for almost all the maximal ideals of $R$ by a similar argument.
So we can build a set of global generators in the following way: with the Chinese Remainder Theorem take $x_2\in I\setminus (\p I+(x_1))$ for every $\p\in T$, $x_3\in I\setminus (\p I+(x_1,x_2))$ for every $\p\in T$ such that $I$ is not equal to $\p I+(x_1,x_2)$, and so on until $x_d$.
Then observe that $x_1,x_2,\cdots,x_d$ is a set of generators for $I$, because it is so locally at every prime: if $\p\in S$ then $I_\p=(x_1,x_2,\cdots,x_d)$ by construction, if $\p\in T\setminus S$ then $I_\p=(x_2)$ and if $\p\not\in T$ then $I_\p=(x_1)$. Now observe that $I=\bigcap_\p I_\p$ and so $x_1,x_2,\cdots,x_d$ generates the ideal $I$ over $R$.
\end{proof}

\begin{corollary}
Let $R$ be a number ring. If $g(R_\p)>1$ for some maximal $R$-ideal $\p$ then
\[g(R) = \sup_\p g(R_\p).\]
\end{corollary}

\begin{remark}
Let $R$ be a number ring such that $g(R_\p)=1$ for every maximal ideal, then $R$ is a Dedekind domain because every ideal $I$ has principal localizations, hence $I$ is invertible. Similarly as in the proof of the previous Lemma, we can show that $g(R)\leq 2$.
\end{remark}

Now that we have introduced some notation, we can start with the proof of Theorem~\ref{thm:firstimpl}.

\begin{lemma}
\label{lemma:linalg1}
Let $U,V,W$ be vector spaces over a field $k$, with $W$ of dimension $\geq 2$. Let $\vphi:U\otimes V \twoheadrightarrow W$ be a surjective linear map. Then there exists an element $u\in U$ such that $\dim_k \vphi(u\otimes V)\geq 2$, or there exists an element $v\in V$ such that $\dim_k \vphi(U\otimes v)\geq 2$.
\end{lemma}
\begin{proof}
For contradiction, assume that $\vphi(u\otimes V)$ and $\vphi(U\otimes v)$ have dimension $\leq 1$, for every choice of $u\in U$ and $v\in V$. As $\vphi$ is surjective, $\{\vphi(u\otimes v):u\in U, v\in V\}$ is a set of generators of $W$. 
Since $W$ has dimension $\geq 2$, among these generators there are 2 which are linearly independent, say $w_1=\vphi(u_1\otimes v_1)$ and $w_2=\vphi(u_2\otimes v_2)$.
Observe
\[\vphi(u_1\otimes v_2)\in \vphi(u_1\otimes V)\cap \vphi(U\otimes v_2)=kw_1\cap kw_2 = 0.\]
Similarly we obtain also $\vphi(u_2\otimes v_1)=0$.
But then we have that both $\vphi((u_1+u_2)\otimes v_1)=w_1$ and  $\vphi((u_1+u_2)\otimes v_2)=w_2$ are in $\vphi((u_1+u_2)\otimes V)$. So it contains two linearly independent vectors and then it must have dimension $\geq 2$. Contradiction.
\end{proof}

\begin{lemma}
\label{lemma:cyclic}
Let $R$ be a commutative domain and $I,J\subset R$ two non-zero ideals, such that $IJ$ can be generated by $3$ elements. Let $\frm \subset R$ be a maximal ideal.
Then one of the following occurs:
\begin{enumerate}
 \item \label{lemma:cyclic:item:1} there exists a non-zero $x\in I_\frm$ such that $(IJ)_\frm/xJ_\frm$ is a cyclic $R_\frm$-module generated by an element of the form $\overline{ij}$ with $i\in I_\frm,j\in J_\frm$;
 \item \label{lemma:cyclic:item:2}there exists a non-zero $y\in J_\frm$ such that $(IJ)_\frm/yI_\frm$ is a cyclic $R_\frm$-module generated by an element of the form $\overline{ij}$ with $i\in I_\frm,j\in J_\frm$.
\end{enumerate}
In particular, if (\ref{lemma:cyclic:item:1}) holds then the morphism of $R_\frm$-modules  induced by the ``multiplication by $j$"
\[\dfrac{I_\frm}{xR_\frm}\overset{\cdot j}{\longrightarrow}\dfrac{(IJ)_\frm}{xJ_\frm}\]
is surjective, and similarly if (\ref{lemma:cyclic:item:2}) holds then the morphism of $R_\frm$-modules  induced by the ``multiplication by $i$"
\[\dfrac{I_\frm}{yR_\frm}\overset{\cdot i}{\longrightarrow}\dfrac{(IJ)_\frm}{yI_\frm}\]
is surjective.
\end{lemma}
\begin{proof}
Let $k$ denote the field $R/\frm$. Observe that $W=(IJ)_\frm/\frm (IJ)_\frm$ is a $k$-vector space of dimension $\leq 3$. First, if $W$ has dimension $1$ then we have that $(IJ)_\frm$ is a principal $R_\frm$-ideal and clearly there exists $x\in I_\frm$ such that $(IJ)_\frm/xJ_\frm$ is a cyclic $R_\frm$-module.
If the dimension of $W$ is $2$ or $3$, then consider the product map:
\[
\vphi:\dfrac{I_\frm}{\frm I_\frm}\otimes \dfrac{J_\frm}{\frm J_\frm}  \longrightarrow W, \qquad
\overline{i}\otimes \overline{j}  \longmapsto \overline{ij}.
\]
It is a surjective linear map of $k$-vector spaces. By Lemma \ref{lemma:linalg1} there exists $x\in I_\frm$  such that $\vphi(x\otimes (J_\frm/\frm J_\frm))$ has dimension $\geq 2$, or there exists $y\in J_\frm$  such that $\vphi( (I_\frm/\frm I_\frm)\otimes y)$ has dimension $\geq 2$. We will prove that if we are in the first case then (\ref{lemma:cyclic:item:1}) holds. The proof that the second case implies (\ref{lemma:cyclic:item:2}) is analogous. So assume that $\dim_k \vphi(x\otimes (J_\frm/\frm J_\frm))\geq 2$. Hence the quotient space
\[	 \dfrac{W}{\vphi(x\otimes (J_\frm/\frm J_\frm))}
\simeq \dfrac{(IJ)_\frm}{xJ_\frm+\frm(IJ)_\frm} \]
has dimension $\leq 1$. Moreover, it is easy to see that it is isomorphic to $S/\frm S$, where $S=(IJ)_\frm/xJ_\frm$. Hence we have that $S$ is a cyclic $R_\frm$-module.\\
We can be more precise saying that every generator of $S$ is of the form $ \sum_{t\in T}\overline{i_tj_t}$, where $T$ is a finite set of indexes, $i_t\in I_\frm$ and $j_t\in J_\frm$. In particular $\set{\overline{i_tj_t}}_{t\in T}$ is a finite set of generators for $S$. As the $k$-vector space $S/\frm S$ is $1$-dimensional,  among the projections $\overline{i_tj_t}$ there exists one $\overline{i_{t_0}j_{t_0}}$ which is a basis. Hence $i_{t_0}j_{t_0}$ is a generator of $S$. The last assertion follows immediately.
\end{proof}

\begin{proposition}
\label{prop:length}
Let $R$ be a commutative Noetherian $1$-dimensional domain. Let $I,J$ be two non-zero ideals such that $IJ$ can be generated by $3$ elements. Then we have that 
\[l\left( \dfrac{R_\frm}{I_\frm} \right) + l\left( \dfrac{R_\frm}{J_\frm} \right) \leq l\left( \dfrac{R_\frm}{(IJ)_\frm} \right). \]
\end{proposition}
\begin{proof}
Assume that case (\ref{lemma:cyclic:item:1}) of Lemma \ref{lemma:cyclic} holds. Consider the ring $R_\frm/xJ_\frm$. It has finite length because it is Noetherian and zero-dimensional. Consider the following diagram of inclusions of $R_\frm$-ideals:
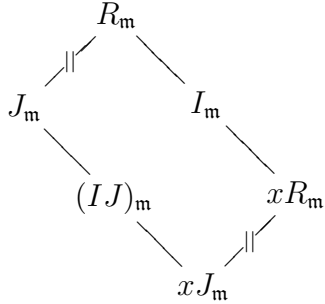
\begin{displaymath}
\xymatrix@ur@C=2pc{
J_\frm \ar@{-}[d] 		& R_\frm \ar@{-}|{||}[l] \ar@{-}[d] \\
(IJ)_\frm \ar@{-}[d] 	& I_\frm \ar@{-}[d] 		 		\\
xJ_\frm 				& xR_\frm \ar@{-}|{||}[l]
}
\end{displaymath}
These two chains define two series for $R_\frm/xJ_\frm$, and they can be refined to composition series. Observe that the multiplication by $x$ is an isomorphism of $R_\frm$ onto $xR_\frm$ and of $J_\frm$ onto $xJ_\frm$, so it induces a $R$-module isomorphism also on the quotients. Hence we have $l (R_\frm/J_\frm) = l (xR_\frm/xJ_\frm) $ and as the diagram of inclusions is commutative we have also $l(R_\frm/xR_\frm)=l(J_\frm/xJ_\frm)$. Moreover, as $I_\frm/xR_\frm$ is mapped onto $(IJ)_\frm/xJ_\frm$ by Lemma \ref{lemma:cyclic}, for every  factor of the composition series between $R_\frm$ and $I_\frm$ there exists a corresponding factor between $J_\frm$ and $(IJ)_\frm$. So we have
\[l\left( \dfrac{R_\frm}{I_\frm} \right) \leq l\left( \dfrac{J_\frm}{(IJ)_\frm} \right).\]
To finish the proof, it is sufficient to add $l(R_\frm/J_\frm)$ on both sides. If case (\ref{lemma:cyclic:item:2}) of Lemma \ref{lemma:cyclic} holds we get the same conclusion with a similar argument.
\end{proof}
Now we can conclude our proof:
\begin{proof}[Proof of Theorem \ref{thm:firstimpl}]
As every ideal can be generated by $3$ elements, for every pair of non-zero $R$-ideals $I$ and $J$, Proposition \ref{prop:length} implies
\[\#(R/\frm)^{l(R_\frm/(IJ)_\frm)}\geq \#(R/\frm)^{l(R_\frm/I_\frm)+l(R_\frm/J_\frm)},\]
for every maximal $R$-ideal $\frm$. Hence by Proposition \ref{prop:lengthproduct} we get
\[N(IJ)\geq N(I)N(J).\]
For the second statement, use Remark \ref{rmk:extension}
\end{proof}
%%%%%%%%%%%%%%%%%%%%%%%%%%%%%%%%%%%%%%%%%%%%%%%%%%%%%%%%%%%%%%%%%%%%%%%%%%%%%%%%%%%%%%%%%%%%%%%%%%%%%%%%%%%%%%%%%%%%%%%%%%%%%
%%%%%%%%%%%%%%%%%%%%%%%%%%%%%%%%%%%%%%%%%%%%%%%%%%%%%%%%%%%%%%%%%%%%%%%%%%%%%%%%%%%%%%%%%%%%%%%%%%%%%%%%%%%%%%%%%%%%%%%%%%%%%
%%%%%%%%%%%%%%%%%%%%%%%%%%%%%%%%%%%%%%%%%%%%%%%%%%%%%%%%%%%%%%%%%%%%%%%%%%%%%%%%%%%%%%%%%%%%%%%%%%%%%%%%%%%%%%%%%%%%%%%%%%%%%
%%%%%%%%%%%%%%%%%%%%%%%%%%%%%%%%%%%%%%%%%%%%%%%%%%%%%%%%%%%%%%%%%%%%%%%%%%%%%%%%%%%%%%%%%%%%%%%%%%%%%%%%%%%%%%%%%%%%%%%%%%%%%
%%%%%%%%%%%%%%%%%%%%%%%%%%%%%%%%%%%%%%%%%%%%%%%%%%%%%%%%%%%%%%%%%%%%%%%%%%%%%%%%%%%%%%%%%%%%%%%%%%%%%%%%%%%%%%%%%%%%%%%%%%%%%

\section{Proof of Theorem \ref{thm:secondimpl}}
\label{sec:secondimpl}
In this section we prove Theorem \ref{thm:secondimpl}.
Firstly, we will exhibit a bound for $g(R)$ for a local number ring $R$.
Secondly, we will give a sufficient condition such that this bound is $\leq 3$.
Finally, we will conclude the proof by moving from the local case to the global one.
\begin{lemma}
\label{lemma:linalg2}
Let $V$ be a finite dimensional vector space over a finite field $k$ such that 
\[V=V_1\cup\cdots\cup V_n,\]
where each $V_i$ is a proper subspace of $V$. Then $n>\# k$.
\end{lemma}
\begin{proof}
As $V_i\subsetneqq V$, then it has codimension $\geq 1$, which implies that $\#V_i\leq(\#k)^{\dim_k V-1}$. As $\overline{0}\in V_1\cap\cdots\cap V_n$, then
\begin{align*}
(\#k)^{\dim_k V}=\#V &=\#(V_1\cup\cdots\cup V_n)\leq \\
&\leq \left(\sum_{i=1}^n \#V_i\right)-(n-1)<\sum_{i=1}^n\#V_i\leq n(\#k)^{\dim_k V-1}. 
\end{align*}
Then dividing by $(\#k)^{\dim_k V-1}$ we get $n>\#k$.
\end{proof}

\begin{lemma}
\label{lemma:2}
Let $R$ be a local number ring with maximal ideal $\frm$ and residue field $k$. Let $\tilde{R}$ be its normalization. Let $l$ be the number of distinct maximal ideals of $\tilde{R}$. If $l\leq \#k$ then for every $R$-ideal $I$, there exists $x\in I$ such that $I\tilde{R}=x\tilde{R}$.
\end{lemma}
\begin{proof}
The statement is trivially true if $I=0$. 
Assume that $I\neq 0$.
Denote the maximal ideals of $\tilde{R}$ by $\frm_1,\cdots,\frm_l$.
Consider the $k$-vector spaces $W=I/\frm I$ and $I\tilde{R}/\frm_iI\tilde{R}$.
For every $i$, define the map
\[\vphi_i: W \longrightarrow \dfrac{I\tilde{R}}{\frm_iI\tilde{R}}, \qquad x+\frm I \mapsto x+\frm_iI\tilde{R}.\]
Denote by $W_i$ the kernel of $\vphi_i$. The ideal $I$ is a set of generators of $I\tilde{R}$ as $\tilde{R}$-module and hence of $I\tilde{R}/\frm_iI\tilde{R}$. This means that $\vphi_i$ is not the zero map, i.e.~$W_i$ is a proper subspace of $W$, for every $i$. As $l\leq \#k$, by Lemma \ref{lemma:linalg2} we get that $W_1\cup \cdots \cup W_l\subsetneq W$ and so there exists $x\in I$ whose projection in $W$ is not in $W_i$, for every $i$. 
Observe that this condition means that $\ord_{\frm_i}(x)\leq \ord_{\frm_i}(I\tilde{R})$ for every $i$. 
Moreover $x\in I\subset I\tilde{R}$, so $\ord_{\frm_i}(x)\geq \ord_{\frm_i}(I\tilde{R})$ for every $i$. 
Since we have that $\ord_{\frm_i}(x) = \ord_{\frm_i}(I\tilde{R})$ for every $i$, we conclude that $x\tilde{R}=I\tilde{R}$.
\end{proof}
The next example proves that the hypothesis $l\leq \#k$ in the previous lemma cannot be omitted. It is a generalization of an example suggested by Hendrik Lenstra.

\begin{example}
 Let $p$ be a prime number and $K$ an extension of $\Q$ of degree $p+1$ where $p$ splits completely. 
 Let $A$ to the integral closure of $\Z_{(p)}$ in $K$. 
 Note that $p$ factors in $A$ as $pA=\q_1\q_2\cdots \q_{p+1}$ and $A/pA$ is isomorphic to the product of $p+1$ copies of $k=\F_p$. Let $R=\Z_{(p)}+pA$. It is a local subring of $A$ with integral closure $A$ and unique maximal ideal $pA$. Consider the surjective morphism $\vphi:A\to A/pA \tilde\to k^{p+1}$. As $R$ contains the kernel of $\vphi$ and $\vphi(R)\simeq k$
 %\vphi(R)=\{(r,r,\ldots,r):r\in k\}, because has #k elements and it must contain the additive subgroup generated by (1,1,\ldots,1).%
 we see that $R=\vphi^{-1}(\{(r,r,\ldots,r):r\in k\})$.
 This implies that the preimage under $\vphi$ of any additive subgroup of $k^{p+1}$ is a fractional $R$-ideal.
Define $J$ as the preimage of the additive subgroup generated by the elements $(1,0,1,1,...,1)$ and $(0,1,1,2,3,...,p-1)$. Observe that every element of $J \mod pA$ has the form $(x,y,x+y,x+2y,x+3y,...,x+(p-1)y)$ for some $0\leq x,y\leq p-1$ and hence it has a coordinate equal to $0$.
Moreover, for every index $i=1,...,p+1$ there exists an element with the $i$-th coordinate non-zero.
In particular we have $JA=A$ and hence $JA$ is a pricipal $A$-ideal generated by any unit, say $u$.
Observe that the coordinates of $\vphi(u)$ are all non-zero and so $u$ cannot be in $J$.
We conclude that  $J$ is a fractional $R$-ideal whose extension to $A$ cannot be generated by an element of $J$. 
\end{example}

\begin{lemma}
\label{lemma:4}
Let $R$ be a local number ring with maximal ideal $\frm$, residue field $k$ and normalization $\tilde{R}$. Let $l$ be the number of distinct maximal $\tilde R$-ideals. If $l\leq \#k$ then for every $R$-ideal $I$ we have that
\[\dim_k\dfrac{I}{\frm I}\leq \dim_k \dfrac{\tilde{R}}{\frm \tilde{R}}.\]
\end{lemma}
\begin{proof}
By Lemma \ref{lemma:2} we know that $\frm\tilde{R}=x\tilde{R}$ for some $x\in \frm$. Since the additive groups of $\tilde R/I$ and $x\tilde R/xI$ are isomorphic, we have
\[[\tilde{R}:\frm\tilde{R}]=[\tilde{R}:x\tilde{R}]=[I:xI]=[I:\frm I][\frm I:xI].\]
In particular, $[I : \frm I]$ divides $[\tilde{R}:\frm\tilde{R}]$, and as $I/\frm I$ and $\tilde{R}/\frm \tilde{R}$ are both $k$-vector spaces we get our statement on their $k$-dimensions.
\end{proof}

Now we would like to drop the hypothesis on the size of the residue field.
The construction described in the proof of the next theorem allows us to enlarge the residue field without losing information on the number of generators of any ideal. Compare with \cite{delcdvor00}[Section 3.1].

\begin{theorem}
\label{thm:local3gp}
Let $R$ be a local number ring with maximal ideal $\frm$, residue field $k$ and normalization $\tilde{R}$. Then for every $R$-ideal $I$ we have that
\[\dim_k\dfrac{I}{\frm I}\leq \dim_k \dfrac{\tilde{R}}{\frm \tilde{R}}.\]
\end{theorem}
\begin{proof}
We want to apply Lemma \ref{lemma:4}. Let $\frm_1,\cdots,\frm_l$ be the distinct maximal ideals of $\tilde{R}$ which are above $\frm$.
Choose $\overline{f}$, a monic irreducible polynomial in $\F_p[X]$ of degree $d$ coprime with $[(\tilde{R}/\frm_i):\F_p]$ for every $i=1,\cdots,l$ and such that $(\#k)^d\geq l$.
Observe that such $d$ is also coprime with $[k:\F_p]$ because each $\tilde{R}/\frm_i$ is a field extension of $k$.
Let $f$ be a monic lift of $\overline f$ to $\Z[X]$.
Note that $f$ is irreducible and of degree $d$.
%Let $\alpha$ be a zero of $f$ and consider $\Q(\alpha)$. It is a number field of degree $d$ over $\Q$ and let $S$ be the order $\Z[X]/(f)$.
Consider the order $S=\Z[X]/(f)$.
We know that as $f$ is irreducible modulo $p$ the prime  $p$ is inert in $S$.
In particular, 
\[ \frac{S}{pS}\simeq \frac{\F_p[X]}{(\overline f)}\simeq \F_{p^d}.\]
Define $T=R\otimes_{\Z} S$ and observe that $ T\simeq R[X]/(f) $.
% ARGUMENT SIMPLIFIED AFTER THE REVIEW
% We will now prove that $T$ is a local domain with unique maximal ideal $\frm \otimes_\Z S=\frM$.
% First of all observe that the ring $k\otimes_{\F_p}(S/pS)$ is a field. Indeed, if we consider the quotient
% \[F:=\dfrac{k\otimes_{\F_p}(S/pS)}{\frR},\]
% where $\frR$ is a maximal ideal, then clearly $F$ is a field extension of $\F_p$. As both $k$ and $S/pS$ can be embedded in $F$ the degree $[F:\F_p]$ is divisible by $[k:\F_p]d$ because they are coprime. This is exactly the dimension of $k\otimes_{\F_p}(S/pS)$ as $\F_p$-vector space. This means that $F=k\otimes_{\F_p}(S/pS)$, hence it is a field. 
% Now observe that $T/\frM$ is a $\F_p$-vector space and that
% \[k \otimes_{\F_p} \dfrac{S}{pS} \simeq \dfrac{R \otimes_\Z S}{(\frm \otimes_\Z S) +( R\otimes_\Z pS)} \simeq \dfrac{T}{\frM}.\]
% So $T/\frM$ is a field and $\frM$ is a maximal ideal. To prove that it is the unique one, let $\frN$ be any maximal ideal of $T$ and recall that $T\simeq R[X]/(f)$. So $T$ is a finitely generated $R$-module and hence $T$ is integral over $R$. This means that $\frN\cap R$ must be the maximal ideal $\frm$, that is $\frN$ contains the $T$-ideal generated by $\frm$, which is $\frM$, and by maximality they are equal. Therefore $T$ is local.
We have
\[ \dfrac{S}{pS} \otimes_{\F_p} k \simeq \dfrac{k[X]}{(\tilde f)} \simeq \dfrac{R[X]}{(\frm,f)}, \]
where $\tilde f$ is the image of $f$ in $k[X]$.
Since the degrees of $S/pS$ and $k$ over $\F_p$ are coprime, we deduce that $\tilde f$ is irreducible in $k[X]$ and $(\frm,f)$ is a maximal ideal of $R[X]$.
Since the maximal ideals of $T$ are in bijection with the maximal ideals of $R[X]$ containing $(\frm,f)$, we deduce that $T$ is a local domain.
We will denote its unique maximal ideal by $\frM$.
% Now observe that also $\tilde{R}/\frm_i$ and $S/pS$ have coprime degree over $\F_p$ and that $\tilde{R}\otimes_\Z S = \tilde{R}[X]/(f)$ is integral over $\tilde{R}$. By the same argument as before we can deduce that there exists an isomorphism of fields
% \[\dfrac{\tilde{R}}{\frm_i} \otimes_{\F_p} \dfrac{S}{pS} \simeq \dfrac{\tilde{R}\otimes_\Z S}{\frm_i \otimes_\Z S},\]
% and that the maximal ideals of $\tilde{R}\otimes_\Z S$ are exactly the $\frm_i\otimes_\Z S=\frM_i$, with $i=1,\cdots,l$.
% The ring $\tilde{R}$ is a semilocal Dedekind domain, so each of its maximal ideals $\frm_i$ is principal by Proposition \ref{prop:semilocal}. Then also each $\frM_i$ is principal, hence invertible, and we have that $\tilde{R}\otimes_\Z S$ is a Dedekind domain, hence it is equal to $\tilde{T}$.

Let $\tilde{T}$ be normalization of $T$.
With a similar argument we can show that $\tilde T$ is semilocal with maximal ideals corresponding to the maximal ideals $(\frm_i,f)$ of $\tilde R[X]$, for $i=1,\ldots,l$.

Observe that $T/\frM$ has $(\# k)^d$ elements, which is bigger than $l$.
Then we can apply Lemma~\ref{lemma:4} and we get
\[\dim_{(T/\frM)}\dfrac{I\otimes_\Z S}{\frM(I\otimes_\Z S)}\leq \dim_{(T/\frM)}\dfrac{\tilde{T}}{\frM\tilde{T}}.\]
Now observe that $I\otimes_\Z S= I\otimes_R T$ and using the canonical isomorphisms of tensor products we get
\[\dfrac{I\otimes_R T}{\frM(I\otimes_R T)}\simeq (I\otimes_R T)\otimes_T \dfrac{T}{\frM}\simeq I\otimes_R \dfrac{T}{\frM}\simeq I\otimes_R k\otimes_k \dfrac{T}{\frM}\simeq \dfrac{I}{\frm I}\otimes_k \dfrac{T}{\frM},\]
so
\[\dim_{(T/\frM)}\dfrac{I\otimes_\Z S}{\frM(I\otimes_\Z S)}=\dim_{(T/\frM)} \left(\dfrac{I}{\frm I}\otimes_k \dfrac{T}{\frM}\right) = \dim_k \dfrac{I}{\frm I}. \]
Similarly we have that
\[\dfrac{\tilde{T}}{\frM\tilde{T}}\simeq \tilde{T}\otimes_T\dfrac{T}{\frM}\simeq (\tilde{R}\otimes_R T)\otimes_T \dfrac{T}{\frM}\simeq \tilde{R}\otimes_R \dfrac{T}{\frM} \simeq \tilde{R}\otimes_R k\otimes _k \dfrac{T}{\frM} \simeq \dfrac{\tilde{R}}{\frm\tilde{R}}\otimes_k \dfrac{T}{\frM},\]
so also
\[\dim_{(T/\frM)}\dfrac{\tilde{T}}{\frM\tilde{T}} = \dim_k \dfrac{\tilde{R}}{\frm\tilde{R}}. \]
Finally, we conclude that
\[ \dim_k \dfrac{I}{\frm I} \leq \dim_k \dfrac{\tilde{R}}{\frm\tilde{R}}. \]
\end{proof}

\begin{corollary}
\label{cor:weakmainthm}
Let $R$ be a local number ring with maximal ideal $\frm$, residue field $k$ and normalization $\tilde{R}$, then $g(R)=\dim_k (\tilde{R}/\frm \tilde{R})$.
\end{corollary}
\begin{proof}
Let $r=\dim_k (\tilde{R}/\frm \tilde{R})$ and let $I$ be any $R$-ideal. By Theorem \ref{thm:local3gp} we obtain that $\dim_k (I/\frm I)\leq r$. As every number ring is Noetherian, we have that $I$ is finitely generated and hence we can apply Nakayama's Lemma to get that $I$ is generated by at most $r$ elements. Hence $g(R)\leq r$. Moreover observe that $\tilde{R}$ is a fractional $R$-ideal and we know that it is generated by exactly $r$  elements, so $g(R)=r$.
\end{proof}

The next theorem is due to Hendrik Lenstra.

\begin{theorem}
\label{thm:lenstra}
Let $k$ be a field and $A$ a $k$-algebra with $\dim_k A\geq 4$. Then exactly one of the following holds:
\begin{enumerate}[label=\upshape(\roman*), leftmargin=*, widest=iii]
\item \label{item:1} there exist $x,y\in A$ such that $\dim_k(k1 + kx + ky +kxy)\geq 4;$
\item \label{item:2} there exists a $k$-vector space $V$ with $A=k\oplus V$ and $V\cdot V=0;$
\item \label{item:3} there exists a $k$-vector space $V$ with $A\simeq
\begin{bmatrix}
k & V\\
0 & k
\end{bmatrix}$, that is $A=ke\oplus kf\oplus V$, with $V\cdot V=eV=Vf=0, e^2=e, f^2=f, ef=fe=0, e+f=1$.
\end{enumerate}
\end{theorem}
\begin{proof}
Suppose that \ref{item:1} does not hold, which means that for every  $x,y\in A$ such that $x\not\in k$ and $y\not\in k+kx$ we have that $xy\in k1 + kx + ky$. First we claim that for every $x\in A$ we have $x^2\in k+kx$. Pick $y\not\in k+kx$. We have $xy\in k1 + kx + ky $ and $x(y+x)\in k1 + kx + k(y+x)=k1 + kx + ky;$ hence $x^2\in k1 + kx + ky$. We can use the same argument for $z\not\in k1 + kx + ky\supset k+kx$ (which exists because the dimension of $A$ over $k$ is $\geq 4$) and we get that $x^2\in k1 + kx + kz$, so $x^2\in (k1 + kx + ky) \cap (k1 + kx + kz)=k + kx$. From these considerations we get that every subspace $W\subset A$ containing $1$ is closed under multiplication, hence it is a ring.

Observe that each $x\in A$ acts by multiplication on the left on $A/(k+kx)$ and each vector is an eigenvector. This means that there is one eigenvalue and hence the action of $x$ is just a multiplication by a scalar. This means that there exists a unique $k$-linear morphism $\lambda: A\longrightarrow k$, such that $xy \equiv \lambda(x) y \mod (k+kx)$ for every $y\in A$. We can use the same argument for the action of $y$ on $A/(k+ky)$ and the action of $xy$ on $A/(k+kx+ky)$, which has dimension $>0$, by hypothesis. As all the actions are scalar on $A/(k+kx+ky)$  we get that $\lambda(x)\lambda(y)=\lambda(xy)$. As this works for every $x,y\in A$ then $\lambda:A\rightarrow k$ is a $k$-algebra morphism. We can use the same argument for the multiplication on the right, to get that there is a unique ring homomorphism $\mu:A\rightarrow k$ such that for every $x,z\in A$ we have $zx\equiv \mu(x)z \mod (k+kx)$. Then we get that $A=k +\ker \lambda = k+\ker \mu$, which also implies that the dimension over $k$ of the 
kernels is $\geq 3$.

Now we want to prove that $\ker\lambda \cdot \ker \mu =0$. For $x\in \ker\lambda$ and $y\in \ker \mu$ we have $xA\subset k+kx$ and $Ay\subset k+ky$. Observe that $xy\in xA\cap Ay$. If $k+kx\neq k+ky$ then $xA\cap Ay \subseteq k$ and as both $\lambda$ and $\mu$ are the identity on $k$ then $xy=\lambda(xy)=\lambda(x)\lambda(y)=0$. Otherwise if $k+kx=k+ky$, pick $z\in \ker\mu \setminus (k+kx)$, which is possible because $\dim_k \ker \mu\geq 3$. Then observe that $(k+kx)\cap (k+kz) = k$, so $xz\in xA\cap Az \subseteq k$. As $\mu$ is the identity on $k$, we have $xz=\mu(xz)=\mu(x)\mu(z)=0$. Similarly $x(y+z)\in xA\cap A(y+z)\subset (k+kx)\cap (k+k(y+z)) = (k+kx)\cap (k+kz) = k$, so also $x(y+z)=\mu(x(y+z))=\mu(x)(\mu(y)+\mu(z))=0$. Hence we get that $xy=0$.

Now we have to distinguish two cases. If $\ker \mu=\ker \lambda$ then, as $\lambda$ and $\mu$ agree on $k$, they coincide on the whole $A$. So we are in case \ref{item:2} with $V=\ker \mu = \ker \lambda$.
If $\ker \mu\neq \ker \lambda$, then call $V=\ker \mu \cap \ker \lambda$ which has exactly codimension $2$: as the kernels are different it must be strictly bigger than $1$ and it is strictly smaller than $3$ because $\ker \mu, \ker \lambda$ have codimension $1$. So the projections of $1,\ker\lambda,\ker\mu $ are 3 distinct lines in $A/V$. Hence: $\ker \lambda =k\cdot e + V$ where we choose $e$ with $\mu(e)=1$ (it can be done as $\mu$ maps surjectively onto $k$), $\ker \mu = k\cdot f + V$ where $f=1-e$. Observe that $ef=e(1-e)=(1-f)f=0$, because $e\in \ker \lambda$ and $f\in \ker \mu$. Then we obtain $e^2=e,f^2=f,fe=0$. Also $eV=Vf=0$. From this conditions we get that $A=ke\oplus kf \oplus V$, because $\ker \lambda= ke\oplus V$ has codimension $1$ and $f\not\in \ker \lambda. $ Then
\begin{align*}
A & \longrightarrow
\begin{bmatrix}
k & V\\
0 & k
\end{bmatrix}\\
ae+bf+v & \longmapsto 
\begin{pmatrix}
b & v\\
0 & a
\end{pmatrix}
\end{align*} 
is a well defined morphism and clearly it is bijective. So we are in case \ref{item:3}.\\
To conclude, observe that if \ref{item:2} holds then $A$ is a commutative algebra and in case \ref{item:3} $A$ is not. If $A$ has \ref{item:2} then it has not \ref{item:1}, because the subspace $k1+kx+ky$ is a ring and so $\dim_k(k1+kx+ky+kxy)\leq 3$. If $A$ has 
\ref{item:3} then it cannot have \ref{item:1}, because if 
$x=
\begin{pmatrix}
a & u\\
0 & b
\end{pmatrix}$ and
$y=
\begin{pmatrix}
c & v\\
0 & d
\end{pmatrix}$ then we have $(x-a)(y-d)=0$ and so $xy\in k+kx+ky$.
\end{proof}

\begin{proposition}
\label{prop:3}
Let $R$ be a local number ring, with maximal ideal $\frm$ and residue field $k$. Assume that $R'= \frm\tilde{R}+R$ is super-multiplicative, where $\tilde{R}$ is the normalization of $R$.
Then 
\[\dim _k \dfrac{\tilde{R}}{\frm \tilde{R}}\leq 3.\]
\end{proposition}
\begin{proof}
Put $A=\tilde{R}/\frm\tilde{R}$. Observe that $A$ is an $R$-module annihilated by the maximal ideal $\frm$, so it is a finite dimensional $k$-algebra. Assume by contradiction that $\dim_k A \geq 4$, so we are in one of the three cases of Theorem \ref{thm:lenstra}. As $\tilde{R}$ is commutative, then $A$ is the same, so we cannot be in case \ref{item:3}.
Assume that we are in case \ref{item:2}, that is $A=k\oplus V$, with $V$ a $k$-vector space such that $V^2=0$. Consider the projection $\tilde{R}\twoheadrightarrow A$ and let $\tilde{\frm}$ be the pre-image of $V$. Observe that $k=A/V\simeq \tilde{R}/\tilde{\frm}$, hence $\tilde{\frm}$ is a maximal ideal of $\tilde{R}$.
The ring $\tilde{R}$ is integrally closed so we have that $\dim_k (\tilde{\frm}/\tilde{\frm}^2)=1$. Therefore also $\dim_k (V/V^2)=\dim_k V =1$ as $V^2=0$. This implies that $\dim_k A=2$. Contradiction. 
Assume that we are in case \ref{item:1}.
Then there exist $\overline{x},\overline{y}\in A$ such that $\dim_k(k1+k\overline{x}+k\overline{y}+k\overline{xy})\geq 4$.
Let $x$ and $y$ be the preimages in $\tilde{R}$ of $\overline{x}$ and $\overline{y}$.
Now consider the $R'$-fractional ideals $I=(1,x,\frm \tilde{R})$ and $J=(1,y,\frm\tilde{R})$. Observe that $\tilde{R}/R'\simeq A/k$ and inside it we have $I/R'$ and $J/R'$ which are generated by the images of $x$ and $y$, respectively, so they corresponds to subspaces of dimension $1$ over $k$. The image of the product $IJ/R'$ is generated by the projections of $x,y$ and $xy$. Therefore it has dimension $\leq 3$ over $k$. Recalling our convention on the index of fractional ideals,  we have 
\[ (\#k)^{3}\geq [IJ:R']>[I:R'][J:R']=(\#k)^{2}.\]
But this contradicts the hypothesis that $R'$ is super-multiplicative. Therefore we must have $\dim_k A \leq 3$.
\end{proof}

Now to conclude the proof of Theorem \ref{thm:secondimpl} stated in the introduction, we need to return to the non-local case.

\begin{proof}[Proof of Theorem \ref{thm:secondimpl}]
Observe that by Lemma \ref{lemma:smlocal} we have that the localization of $R+\frm \tilde R$ at every maximal ideal $\frm$ is super-multiplicative. Then by Proposition \ref{prop:3} and Corollary \ref{cor:weakmainthm} we get that every $R_\frm$-ideal is generated by 3 elements, for every $\frm$. Then by Lemma \ref{lemma:d-gen} we have that every $R$-ideal is generated by 3 elements.
\end{proof}

Let us summarize what we proved: let $R$ be a number ring with normalization $\tilde R$ and consider the ring extensions of $R$ given by $R'(\frm) = R+\frm\tilde R$, where $\frm$ is a maximal ideal of $R$. Then
\begin{displaymath}
    \xymatrix{
g(R)\leq 3 \ar@2{->}[d]\ar@2{->}[r] & g(R'(\frm))\leq 3\ (\forall \frm) \ar@2{->}[d] \\
R \text{ super-mult. } & R'(\frm)\text{ super-mult.} \ (\forall \frm) \ar@2{->}[ul]
        }
\end{displaymath}
We cannot say that all the statement are equivalent because if $R$ is super-multiplicative then it is possible that there exists an extension $R'$ (of the required form) which is not, as we show in the next example, which was communicated by Hendrik Lenstra.
\begin{example}
 Let $p$ be a prime number. Let $\alpha$ be a root of a monic polynomial of degree $4$ with coefficients in $\Z_{(p)}$ which is irreducible modulo $p$. Let $A=\Z_{(p)}[\alpha]$. Observe that $A$ is a local domain with maximal ideal $pA$. Moreover $A$ is Noetherian and has Krull dimension $1$. Therefore $A$ is a discrete valuation ring and so it is integrally closed. Put $R'=\Z_{(p)}\oplus pA$ and $R=\Z_{(p)}\oplus p\alpha\Z_{(p)}\oplus p\alpha^2\Z_{(p)}\oplus p^2\alpha^3\Z_{(p)}$. Observe that $R'$ is the ring of Example \ref{ex:degree4} tensored with $\Z_{(p)}$, hence not super-multiplicative. Moreover, $R$ is a local subring of $R'$ with maximal ideal $\frm=p\Z_{(p)}\oplus p\alpha\Z_{(p)}\oplus p\alpha^2\Z_{(p)}\oplus p^2\alpha^3\Z_{(p)}$, normalization $A$ and residue class field $k=\F_p$. Notice that $R'$ can be described also as $R'=R+\frm A$. We will prove now that $R$ is super-multiplicative.
 
 First we look at the quotient $R/pR$. Let $x$ and $y$ be the images of $p\alpha$ and $p\alpha^2$ under the quotient map. Then $R/pR$ is a $k$-algebra of dimension $4$ with basis $1,x,y$ and $xy$. Moreover $R/pR$ is a local ring with maximal ideal $(x,y)$ and, from the relations $x^2=xy^2=0$, we see that the annihilator of $x$ in $R/pR$ is $kx+kxy$ and the annihilator of $(x,y)$ is $kxy$. Pulling back this statement to $R$, we obtain that $R\cap ((pR):\frm) = pR+p^2\alpha^3\Z_{(p)}$. But $((pR):\frm)$ is contained in $((pR):(pR))=R$, so $((pR):\frm)=pR+p^2\alpha^3\Z_{(p)}$. Dividing by $p$ we get $(R:\frm)=R+p\alpha^3\Z_{(p)}=R'$. In particular $(R:\frm)$ is a ring and $[(R:\frm):R]=p$.
 
 Now take two non-zero fractional $R$-ideals $I$ and $J$. We want to prove that $N(IJ)\geq N(I)N(J)$. Observe that multiplying by non-zero principal ideals of $R$ does not change the problem.
 By Lemma \ref{lemma:2} there exists $s$ in $I$ such that $IA=sA$. Then $R\subseteq (1/s)I \subseteq (1/s)IA = A$ so we can assume that $I$ contains $R$ and is contained in $A$, and similarly for the ideal $J$.
 If $I$ or $J$ equals $R$ the inequality holds (with equality). So we assume that both $I$ and $J$ properly contain $R$. In particular $N(I)$ and $N(J)$ are at most $1/p$.
 Then $I/R$ and $J/R$ are finite non-zero $R$-modules and have therefore a non-trivial piece annihilated by $\frm$.
%  put M=I/R. Then \frm is the Jacobson radical of R and hence by Prop 2.6 from Atiyah MacDonald (Nakayama Lemma) we have \frm*M=M implies M=0. Since M is non-zero it has a non trivial piece annihilated by \frm . That is, there exists i \in I\setminus R such that i*\frm \subseteq R which is the same as saying that I\cap (R:\frm) \neq R.
 Hence $I\cap (R:\frm)$ contains $R$ properly and using the fact that $[(R:\frm):R]=p$ we obtain that $I\supset (R:\frm)$.
%  easy to see with the diagram of inclusion
 The same holds for $J$. Suppose first that $N(I)=1/p$, then $I=(R:\frm)$ and so $IJ=J$. Then the inequality is valid: $N(IJ)=N(J)>N(J)/p=N(I)N(J)$. Likewise if $N(J)=1/p$. It remains to check the case when both $I$ and $J$ have norm at most $1/p^2$. In this case the inclusion $IJ\subset A$ implies $N(IJ)\ge N(A) = 1/[A:R]=1/p^4 \ge N(I)N(J)$, as required.
\end{example}
 
\section*{Acknowledgement}
This article is the result of the work done for my Master Thesis written at the University of Leiden under the supervision of Bart de Smit.
I thank him for his guidance and suggestions.
I am grateful to Hendrik Lenstra for his help in some key passages and for his comments.
I would also like to thank Jonas Bergstr\"om for his help in revising previous versions of the article.
The author would also like to express his gratitude to the Max Planck Institute for Mathematics in Bonn for their hospitality.

\end{document}